\documentclass{amsart}
\usepackage{amsmath, amssymb, amsthm}

\newtheorem{theorem}{Theorem}
\newtheorem{lemma}{Lemma}
\newtheorem{proposition}{Proposition}

\newtheorem{manualtheoreminner}{Theorem}
\newenvironment{manualtheorem}[1]{%
	\renewcommand\themanualtheoreminner{#1}%
	\manualtheoreminner
}{\endmanualtheoreminner}

\theoremstyle{definition}

\theoremstyle{remark}
\newtheorem*{remark}{Remark}

\newcommand{\R}{\mathbb{R}}
\newcommand{\norm}[1]{\left\lVert#1\right\rVert}
\newcommand{\abs}[1]{\left\lvert#1\right\rvert}
\newcommand{\inp}[2]{\left\langle #1, #2 \right\rangle}
\newcommand{\dmu}{\,d\mu}

\DeclareMathOperator{\Tr}{Tr}
\DeclareMathOperator{\Cov}{Cov}
\DeclareMathOperator{\Var}{Var}

\DeclareMathOperator{\vol}{vol}

\title{Sharp stability for the (B)-theorem}
\author{Eli Putterman}
\date{\today}

\begin{document}

\begin{abstract} The (B)-theorem of Cordero-Erausquin, Fradelizi and Maurey states that if $\gamma$ is the standard Gaussian in $\mathbb R^n$, $K \subset \mathbb R^n$ is an origin-symmetric convex set, and $s, t \in \mathbb R$ then $\gamma\left(e^{\frac{s + t}{2}} K\right) \ge \sqrt{\gamma(e^{s} K) \gamma(e^{t} K)}$. Herscovici, Livshyts, Rotem and Volberg proved a stability version of this result, showing that if one has equality up to a factor $(1 + \epsilon)$ in the (B)-inequality for $K$ then the inradius of $K$ must be either ``very large'' or ``very small,'' where the bounds depend on $\epsilon$ and on $n$. We give a new stability estimate which is dimension-free and also yields more precise information about bodies which are near-optimizers of the (B)-inequality. In particular, our results imply that if $\gamma\left(e^{\frac{s + t}{2}} K\right) \le (1 + \epsilon) \sqrt{\gamma(e^{s} K) \gamma(e^{t} K)}$, then every principal component of the covariance matrix of the probability measure obtained by restricting the Gaussian to $K$ must either be at least $1 - O(\epsilon)$ or at most $O(\epsilon)$, which is sharp. Our method extends immediately to yield stability estimates for generalizations of the (B)-inequality, namely the ``strong'' and ``functional'' (B)-inequalities, which reduce to spectral questions about $1$-log-concave measures on $\mathbb R^n$.
\end{abstract}

\maketitle

\section{Introduction}
Fix a natural number $n$, and let $\gamma$ denote the standard Gaussian measure on $\R^n$, which has density $d\gamma = (2\pi)^{-n/2} e^{-\frac{|x|^2}{2}}\,dx$. For a matrix $A \in \mathbb R^{n \times n}$ and $K \subset \mathbb R^n$ we set $AK = \{Ax: x \in K\}$; in particular, for $t \in \mathbb R$ we write $tK = (tI_n)K$. We call a set $K \subset \mathbb R^n$ origin-symmetric, or simply symmetric, if $K=-K := (-1)K$.

By the Pr\'ekopa-Leindler inequality, for any convex set $K \subset \mathbb R^n$ and $a, b > 0$ it holds that 
$$ \gamma\left(\frac{a + b}{2}K\right) \ge \sqrt{\gamma(aK) \gamma(bK)}.$$

It was conjectured by Banaszczyk (as reported by Lata{\l}a in \cite{L02}), and proven by Cordero-Erausquin, Fradelizi and Maurey \cite{CFM03}, that the Pr\'ekopa-Leindler inequality can be strengthened if $K$ is a symmetric convex set:

\begin{theorem}[(B)-theorem]\label{thm:B_thm} For every symmetric convex set $K\subset\R^n$ and every $s, t > 0$,
\begin{equation}\label{eq:B_thm}
\gamma\left(\sqrt{ab} K\right)\geq\sqrt{\gamma(aK)\gamma(bK)}. 
\end{equation}
\end{theorem}

Equivalently, if $K$ is a symmetric convex set then $f(t) = \gamma(e^t K)$ is a log-concave function on $\mathbb R$, i.e., $\log f$ is concave.

In fact, Cordero-Erausquin, Fradelizi and Maurey \cite{CFM03} proved a more general statement, regarding dilations by arbitrary positive semidefinite matrices:

\begin{theorem}[Strong (B)-theorem]\label{thm:strong_B_thm} For every symmetric convex set $K \subset\R^n$ and every symmetric matrix $A$, the function $t \mapsto \gamma\left(e^{tA} K\right)$
is log-concave on $\mathbb R$. Equivalently, for every $s, t \in \mathbb R$, one has
\begin{equation} \label{eq:strong_B}
\gamma\left(e^{\frac{s + t}{2} A} K\right) \ge \sqrt{\gamma(e^{sA} K) \gamma(e^{tA} K)}.
\end{equation}
\end{theorem}

We remark that by the rotational invariance of the Gaussian, this result immediately reduces to the case that $A$ is a diagonal matrix, which is how the theorem is often stated.

The (B)-theorem is a particular case of a still mostly conjectural theory, the log-Brunn-Minkowski theory, which predicts that classical convex-geometric inequalities may be strengthened in the presence of symmetry assumptions; see \cite{HLRV23} for some pointers to the literature.

Recently, Herscovici, Livshyts, Rotem, and Volberg \cite{HLRV23} studied cases of near-equality in Theorems \ref{thm:B_thm} and \ref{thm:strong_B_thm}. Firstly, they proved that there are no ``nontrivial'' equality cases in the (B)-theorem, and, more generally, in the strong (B)-theorem. That is, if equality holds in \eqref{eq:B_thm} (for $s \neq t$), then either $K$ has empty interior (in which case both sides are $0$) or $K = \mathbb R^n$ (in which case both sides are $1$) \cite[Corollary 1.4]{HLRV23}.

As for the strong (B)-theorem, choose an orthonormal basis in which $A$ is diagonal and $\ker A = \mathbb R^k \times \{0\}$ for some $k \in \{0, 1, \ldots, n\}$. Herscovici et al. proved that if equality holds in \eqref{eq:strong_B}, then, either $\gamma(K) \in \{0, 1\}$ as before, or $K$ is of the form $L \times \mathbb R^{n - k}$ for some symmetric convex set $L \subset \mathbb R^k$ \cite[Corollary 1.5]{HLRV23}.

These characterizations of equality cases are actually simple consequences of the main results of \cite{HLRV23}, which are stability versions of the (B)- and strong (B)-theorems. The term ``stability version'' refers to the following type of result. Suppose an inequality achieves equality only for a specific class of objects (``extremizers''). A stability theorem for this inequality asserts that if an object is close to saturating the inequality, it must be quantitatively close, in a suitable metric, to the class of extremizers. Stability versions of convex-geometric inequalities have attracted a great deal of attention in recent years; see, e.g., \cite{BF22,CF20, FGS24,FIL16}.

To state Herscovici et al.'s stability version of the (B)-theorem, recall that a convex body is a bounded convex set with nonempty interior, and that the inradius $r(K)$ of a symmetric convex body $K$ is defined as
$$r(K) = \sup \{r > 0\,|\, \forall x \in \mathbb R^n: |x| < r \Rightarrow x \in K\}.$$

\begin{theorem}[\cite{HLRV23}, Theorem 1.1]\label{thm:hlrv_stability}
Let $s < t$ and let $K \subset \mathbb R^n$ be a symmetric convex body, and suppose that
$$\gamma(e^{\frac{s + t}{2}} K) \leq (1 + \epsilon) \sqrt{\gamma(e^ s K)\gamma(e^t K)}$$
for some $0 < \epsilon < \frac{c}{n^2} (t - s)$, where $c > 0$ is a universal constant. Then either
$$r(K)\geq \frac{1}{e^t} \sqrt{\log \left(\frac{c (t - s)^2}{n \epsilon}\right)},$$ 
or 
$$r(K)\leq \frac{C\sqrt{n}}{e^s} \epsilon^{\frac{1}{n+1}}(t - s)^{-\frac{2}{n+1}},$$
where $C, c > 0$ are universal constants.
\end{theorem}

Herscovici et al. also prove a stability version of the strong (B)-inequality \cite[Theorem 1.5]{HLRV23}, whose statement we omit as it is quite complicated.

The appearance of two distinct cases of near-equality in the (B)-theorem in Theorem \ref{thm:hlrv_stability} has a clear intuitive meaning. Suppose for simplicity that $s$ and $t$ are of order $1$. If $K$ is extremely small, then in a neighborhood of $K$, the Gaussian measure $\gamma$ is not very different from the Lebesgue measure, for which one always has equality in the analogous ``(B)-theorem'', the trivial identity $\vol(a K) \vol(b K) = \vol(\sqrt{ab} K)^2$, so there will be near-equality for the Gaussian measure as well. On the other hand, if $K$ has a very large inradius, then $\gamma(a K), \gamma(b K), \gamma(\sqrt{ab} K)$ are all close to $1$ and so the (B)-inequality must be close to saturated for trivial reasons.

In addition, in the large-inradius regime, the results of \cite{HLRV23} accurately describe the behavior of near-equality cases of the (B)-theorem not just qualitatively, but quantitatively as well. In fact, by taking $K$ to be a symmetric strip $[-R, R] \subset \mathbb R$, Herscovici et al. showed that the dependence of $r(K)$ on $\epsilon$ as $\epsilon \to 0$ cannot be improved beyond $r(K) \gtrsim \sqrt{\log\frac{1}{\epsilon}}$.

However, in the case of ``small'' convex bodies, Herscovici et al. remarked that their bound was unlikely to be sharp. Indeed, one may compute that if $K = r B_2^n$ is the Euclidean ball of radius $r$ around the origin, then 
$$\frac{\gamma(\sqrt K)}{ \sqrt{\gamma(2K) \gamma(K)}} = 1 + \Theta(r^2),$$
as $r \to 0$. That is, the dependence of the inradius on the stability deficit $\epsilon$ in the family of Euclidean balls is $r(K) \sim \sqrt\epsilon$, rather than $r(K) \sim \epsilon^{\frac{1}{n + 1}}$. If this behavior is representative of the class of ``small'' near-optimizers in the (B)-theorem, then the upper estimate on $r(K)$ in Theorem \ref{thm:hlrv_stability} could be substantially improved.

\subsection{The spectral reduction}
The proof of the (B)-theorem itself, as well as that of its stability version, proceeds via reduction of the (B)-inequality \eqref{eq:strong_B} to ``local'' inequalities, i.e., inequalities which concern a single dilate $e^{At} K$ of $K$ rather than dilates corresponding to different values of $t$. 

The idea is quite simple. Fixing a positive-definite matrix $A$, a symmetric convex body $K$, and writing $V(t) = \log \gamma(e^{A t} K)$, Theorem \ref{thm:strong_B_thm} is equivalent to the statement that $V$ is a concave function of $t$ for any such $A$ and $K$. We may assume $K$ has nonempty interior, as otherwise this is trivial.

Writing $K_t = e^{At} K$, one computes directly that $V$ is twice differentiable, with second derivative given by
\begin{align*}
V''(t) &= \frac{1}{\gamma(K_t)} \int_{K_t} \langle x, Ax\rangle^2\,d\gamma - \frac{2}{\gamma(K_t)}\int_{K_t} |Ax|^2\,d\gamma \\ &\qquad -\left(\frac{1}{\gamma(K_t)}\int_{K_t} \langle x, Ax\rangle\,d\gamma \right)^2.
\end{align*}

Letting $\gamma_{K_t}$ be the normalized restriction of the Gaussian to $K_t$, i.e., the measure $d\gamma_{K_t} = \frac{1}{\gamma(K_t)} 1_{K_t} \,d\gamma$, and $f(x) = \langle x, Ax\rangle$, we can write this as
\begin{equation}\label{eq:B_der2}
V''(t) = \Var_{\gamma_{K_t}}[f] - \frac{1}{2} \mathbb E_{\gamma_{K_t}}[|\nabla f|^2].
\end{equation}
Thus, the strong (B)-theorem reduces to showing that for every $A$ and $K$, the expression on the right-hand side of \eqref{eq:B_der2} is nonpositive.

It was classically known that the measure $\gamma_{K_t}$ satisfies a Poincar\'e inequality with constant $1$, i.e., for any (nice) function $g: \mathbb R^n \to \mathbb R$ one has
$$\Var_{\gamma_{K_t}}[g] \le \mathbb E_{\gamma_{K_t}}[|\nabla g|^2].$$
The crucial idea of Cordero-Erausquin, Fradelizi and Maurey was to show that a stronger Poincar\'e inequality holds if one restricts to even functions: if $g$ is even then one has
\begin{equation}\label{eq:even_poin_Kt}
\Var_{\gamma_{K_t}}[g] \le \frac{1}{2}\mathbb E_{\gamma_{K_t}}[|\nabla g|^2].
\end{equation}
(We will recall one proof of both these inequalities in \S \ref{sec:prelims} below.) Applying \eqref{eq:even_poin_Kt} with $g(x) = \langle x, Ax\rangle$ we obtain $V''(t) \le 0$, as desired.

The proof of stability for the (B)-theorem in \cite{HLRV23} also proceeds via the second derivative. Suppose one has 
$$\gamma(e^{\frac{t + s}{2}} K)^2 \le (1 + \epsilon) \gamma(e^s K) \gamma(e^t K)$$
for some small $\epsilon > 0$. Letting $V(t) = \log \gamma(e^t K)$, this implies that
$$2 V\left(\frac{t + s}{2}\right) - V(s) - V(t) \le \epsilon.$$
As $V$ is concave, this implies its second derivative must be small, due to the following elementary lemma:

\begin{lemma}[{\cite[Lemma 7.1]{HLRV23}}] \label{lem:concavity} If $V \in C^2(\mathbb R)$, $s < t$, then
$$V(s) + V(t) - 2 V\left(\frac{s + t}{2}\right)  = \frac{(t - s)^2}{4} \int_{-1}^1 (1 - |r|) V''\left(\frac{(1 - r) s + (1 + r) t}{2}\right)\,dr.$$
\end{lemma}

As $V''(r) \le 0$ for all $r \in [-1, 1]$ and $(1 - |r|)\,dr$ is a probability density on $[-1,1]$, Lemma \ref{lem:concavity} immediately yields the existence of some $r \in [s, t]$ such that $-V''(r) \le \frac{4\epsilon}{(t - s)^2}$, which implies by \eqref{eq:B_der2} that the even Poincar\'e inequality for $\gamma_{e^{Ar} K}$ is nearly saturated by the function $f(x) = |x|^2$. The main technical work in \cite{HLRV23} goes into drawing out the implications of this fact: by following the proof of the even Poincar\'e inequality for $\gamma_{e^r K}$ and drawing out the implications of near-equality at each step, Herscovici et al. derive the dichotomy for the inradius of $e^r K$ in the statement of Theorem \ref{thm:hlrv_stability}.

\subsection{Our results}
The goal of this note is to generalize and strengthen the stability version of the (B)-theorem introduced by Herscovici et al.

The first ingredient in our work is the generalization of the (B)-theorem beyond convex sets, to the setting of log-concave functions on $\mathbb R^n$. (For our purposes, a log-concave function is an upper semicontinuous function $f: \mathbb R^n \to [0, \infty)$ with convex support $K$ such that $\log f$ is concave on $K$. In particular, indicator functions of closed convex sets are log-concave. To avoid trivialities, we shall always assume that the support of $f$ is $n$-dimensional.)

Recall that the strong (B)-theorem states that if $A$ is a symmetric matrix, $V(t) = \log \gamma(e^{A t} K)$ is log-concave on $\mathbb R$. Letting $1_K$ denote the indicator function of $K$, one may write
$$V(t) = \log \gamma(e^{A t} K) = \log \int_{e^{A t} K} \gamma(x)\,dx = \log \int_{\mathbb R^n} 1_K(e^{-A t} x) \gamma(x)\,dx.$$
It turns out that one can replace the indicator function $1_K$ in the integrand with any even log-concave function, yielding a more general inequality for log-concave functions, which was certainly known to the experts but seems to have first appeared explicitly in \cite{CR20}:

\begin{proposition}[Functional strong (B)-inequality]\label{prop:func_B} If $\phi$ is an integrable even log-concave function on $\mathbb R^n$ and $A$ is any symmetric $n \times n$ matrix, the function 
\begin{equation}\label{eq:functional_B}
    V_A(t) = \log \int_{\mathbb R^n} \phi(e^{-A t} x) \gamma(x)\,dx
\end{equation} 
is concave in $t$.
\end{proposition}

The functional version of the (B)-theorem lies in the spirit of a well-trodden line of research in convex geometry studying functional versions of convex-geometric constructions and inequalities. See \cite{M07}, \cite{C17}, or \cite[Chapter 9]{AGMII} for introductions to the vast literature on this subject, which goes by many different names (``functionalization of geometry'', ``the geometric theory of log-concave functions'', etc.).

The proof of Proposition \ref{prop:func_B} is the same as for the ordinary (B)-inequality: one computes in the same way that 
\begin{equation}\label{eq:func_B_der2}
V_A''(t) = \Var_{\mu_t}[f_A] - \frac{1}{2} \mathbb E_{\mu_t}[|\nabla f_A|^2]
\end{equation}
where $f_A(x) = \langle x, Ax\rangle$ as before, and $\mu_t$ is the measure with density 
\begin{equation}\label{eq:normalized_density}
    d\mu_t(x) = \frac{\phi(e^{-A t} x)\, d\gamma(x)}{\int \phi(e^{-A t} y)\, d\gamma(y)}.
\end{equation}

Each $\mu_t$ is an even $1$-log-concave measure, namely a measure whose density with respect to the Gaussian, $\frac{d\mu_t}{d\gamma}$, is log-concave. It turns out that such measures satisfy the same Poincar\'e inequality on the subspace of even functions as $\gamma_{K_t}$ does: namely, for any (nice) even function $g$, one has 
\begin{equation}\label{eq:even_poin_mu_t}
\Var_{\mu_t}[g] \le \frac{1}{2}\mathbb E_{\mu_t}[|\nabla g|^2].
\end{equation}
Applying \eqref{eq:even_poin_mu_t} to $g = f_A$ and using \eqref{eq:func_B_der2} shows that $V_A''$ is concave, proving Proposition \ref{prop:func_B}. (Again, the proof of \eqref{eq:even_poin_mu_t} will be recalled in \S \ref{sec:prelims}.)

We work in the setting of the functional (B)-theorem not simply because it is more general, but because it offers significant technical advantages: when $\mu_t$ has a smooth density of full support, there are no boundary terms when integrating by parts and no boundary conditions to take into account when solving PDEs. Standard approximation arguments can then be used to transfer results to more general measures, in particular to the Gaussian measures restricted to convex bodies which arise in the original (B)-theorem.

Our stability result on the (functional) strong (B)-theorem proceeds via the same general approach as that of Herscovici et al.: namely, given an even log-concave function $\phi(x)$ such that $V_A(t)$, as defined in \eqref{eq:functional_B}, satisfies
$$2 V_A\left(\frac{t + s}{2}\right) - V_A(t) - V(s) \le \epsilon,$$
combining \eqref{eq:func_B_der2} and Lemma \ref{lem:concavity} immediately implies the existence of $r \in [s, t]$ such that
\begin{equation}\label{eq:func_stab_spec}
\frac{1}{2}\mathbb E_{\mu_r}[|\nabla f_A|^2] - \Var_{\mu_r}[f_A] \le \frac{4\epsilon}{(t - s)^2},
\end{equation}
where $\mu_r$, defined as in \eqref{eq:normalized_density}, is an even, $1$-log-concave measure. Our main results show that \eqref{eq:func_stab_spec} places strong constraints on the covariance matrix of $\mu_r$. 

Our first result is a stability result for the ``weak'' (B)-theorem, which reduces by the above to the situation in which $f(x) = x^2$ is a near-optimizer in the even Poincar\'e inequality for an even $1$-log-concave measure.

\begin{theorem}\label{thm:cov_stability} Let $\mu$ be an even $1$-log-concave measure with covariance matrix $\Sigma$, let $f(x) = |x|^2$, and suppose that 
\begin{equation}\label{eq:spectral_stability}
\frac{1}{2} \mathbb E[|\nabla f|^2] - \Var_\mu(f) \le \delta
\end{equation}
for $\delta \in (0, c)$, where $c > 0$ is an absolute constant. Then
\begin{equation}
\Tr(\Sigma (I - \Sigma)) \le C\delta
\end{equation}
for some universal constant $C > 0$.
\end{theorem}

Note that by the Poincar\'e inequality for $\mu$ (applied to linear functions), one always has $\Sigma \le I$. Hence, by diagonalizing $\Sigma$ the conclusion of Theorem \ref{thm:cov_stability} is equivalent to the following: letting $\sigma_1, \ldots, \sigma_n$ be the eigenvalues of $\Sigma$, one has
$$\sum_{i = 1}^n \min(\sigma_i, 1 - \sigma_i) \le C\delta$$
for a universal constant $C > 0$. 

Also note that Theorem \ref{thm:cov_stability}
manifests the same kind of dichotomy as already seen in the stability result of \cite{HLRV23}, with the small-covariance (resp. large-covariance) regimes corresponding to the case in which $K$ has small (resp. large) inradius.

Our second main result applies to the more general setting of the strong (B)-theorem:

\begin{theorem}\label{thm:strong_cov_stability} Let $\mu$ be an even $1$-log concave measure, $A \in \mathbb R^{n \times n}$ a symmetric matrix, and let $f_A(x) = \langle x, Ax\rangle$. If 
\begin{equation}
\frac{1}{2}\mathbb E_\mu[|\nabla f_A|^2] - \Var_\mu(f_A)  \le \delta,
\end{equation}
then
$$\|A(I - \Sigma) \Sigma^{1/2}\|_{HS}^2 \le C\delta$$
for a universal constant $C > 0$.
\end{theorem}

Our results improve upon Theorem \ref{thm:hlrv_stability} in several ways. Firstly, they apply in the more general setting of the functional (B)-theorem: that is, we characterize when any even $1$-log-concave measure nearly saturates the (functional) (B)-inequality, not just measures which come from restricting the Gaussian to convex bodies. Second, our results are stated in terms of the covariance, which is a finer and perhaps more natural invariant than the inradius. In particular, Theorem \ref{thm:cov_stability} excludes more cases than Theorem \ref{thm:hlrv_stability}: for example, it implies that $K = [-\epsilon, \epsilon] \times [1, 1]$ is not $\epsilon$-close to saturating the (B)-inequality, whereas Theorem \ref{thm:hlrv_stability} does not: indeed, $K$ has small inradius, but the covariance of $\gamma_K$ along the $y$-axis is bounded away from $0$ and from $1$. In addition, our bounds are dimensionless (i.e., the constants in Theorems \ref{thm:cov_stability} and \ref{thm:strong_cov_stability} do not depend at all on the dimension $n$). Finally, as we shall see in \S \ref{sec:sharp_impr}, Theorem \ref{thm:cov_stability} it is asymptotically sharp.

\begin{remark} While log-concave functions and measures serve as a natural generalization of convex sets, a much farther-reaching generalization of the theory of log-concave measures is the theory of metric measure spaces under a curvature lower bound. From the perspective of this theory, a $1$-log-concave measure $\mu$ on $\mathbb R^n$ is simply an instance of a metric measure space satisfying the curvature condition $RCD(1, \infty)$, and many of the useful properties of $(\mathbb R^n, \mu)$ such as functional and isoperimetric inequalities, hold for all $RCD(1, \infty)$ spaces. (See \cite{S18} for a gentle introduction to metric measure spaces, curvature lower bounds, and their many applications, including versions of the Brunn-Minkowski and Pr\'ekopa-Leindler inequalities.) In particular, $RCD(1, \infty)$ spaces satisfy a Poincar\'e inequality with constant $1$, just like $1$-log-concave measures.

The (B)-theorem has not, thus far, been generalized to the metric measure space setting; indeed, even the statement of the theorem relies on the linear structure of $\mathbb R^n$ via the notions of convex sets (or log-concave functions) and dilates. One can instead consider the purely spectral version of the (B)-theorem, and ask what properties must an $RCD(1, \infty)$ metric space have if its Poincar\'e constant is very close to $1$. Bertrand and Fathi \cite{BF22} studied precisely this question, and proved a stability result for the Poincar\'e inequality on general $RCD(1, \infty)$ spaces. (Courtade and Fathi \cite{CF20} had earlier studied the stability of the Poincar\'e inequality for $1$-log-concave measures on $\mathbb R^n$.)

To generalize the \textit{even} Poincar\'e inequality for even $1$-log-concave measures, one would need to show how to improve the Poincar\'e inequality under symmetry assumptions in the general setting of $RCD(1, \infty)$ spaces. We are not aware of any work in this direction. 
\end{remark}

\subsection{Outline of the proof}\label{subsec:method_of_proof}
The proof of Theorems \ref{thm:cov_stability} and \ref{thm:strong_cov_stability} consists of two main steps. The first is to follow the proof of the even Poincaré inequality carefully and determine the consequences of near-equality at each step of the argument. Applying this ``reverse engineering'' method yields Theorem \ref{thm:strong_cov_stability} in an entirely natural manner. (We remark that Herscovici et al. \cite{HLRV23} employed a similar strategy, but our argument differs from theirs in the details of the analysis.) 

Applying Theorem \ref{thm:strong_cov_stability} in the case $A = I$ yields that, if 
$$2\cdot \mathbb E_{\mu}[|x|^2] - \Var_{\mu}(|x|^2) \le \delta,$$ 
then each principal covariance of $\mu$ is either bounded above by $O(\delta)$ or bounded below by $1 - O(\sqrt \delta)$. The bound in the small-covariance regime is optimal, but the bound in the large-covariance regime is not: the correct lower bound, as demonstrated by Theorem \ref{thm:cov_stability}, is $1 - O(\delta)$.

Thus, to obtain optimal stability bounds in the large-variance regime and complete the proof of Theorem \ref{thm:cov_stability}, we require an additional argument. We prove a slightly generalized version of Klartag's improved Lichnerowicz inequality from \cite{K23} and combine it with some of the spectral information we extracted ``along the way'' in the proof of Theorem \ref{thm:strong_cov_stability} and use some additional linear-algebraic tricks to obtain a bound of $1 - O(\delta)$ on the large-variance directions. This part of the argument seems somewhat less natural than the first step, but we do not know how to use the ``reverse engineering'' method to obtain Theorem \ref{thm:cov_stability} directly.

In \S \ref{sec:prelims}, we give the necessary background, in particular recalling the proofs of the Poincar\'e and even Poincar\'e inequalities for even $1$-log-concave measures. The proofs of Theorem \ref{thm:strong_cov_stability} and of Theorem \ref{thm:cov_stability} are then given in \S \ref{sec:init_stability} and \S \ref{sec:sharp_impr}, respectively.

\subsection*{Acknowledgements} I would like to thank Alexandros Eskenazis and Liran Rotem for helpful discussions. Part of this work was conducted while the author was supported by a Chateaubriand Fellowship.

\section{Preliminaries}\label{sec:prelims}

\subsection*{Notations} As above, we use $\|\cdot\|_{HS}$ to denote the Hilbert-Schmidt or Frobenius norm on the space $\mathbb R^{n \times n}$ of $n \times n$ matrices, defined by $\|A\|_{HS}^2 = \sum_{i, j = 1}^n A_{ij}^2$. 

For vectors $x, y \in \mathbb R^n$, the matrix $x \otimes y \in \mathbb R^{n \times n}$ is defined as $(x \otimes y)_{ij} = x_i y_j$.

For a measure $\mu$ on $\mathbb R^n$, we write $\Cov(\mu)$ for its covariance matrix, which, if $\mu$ is even, is simply
$$\Cov(\mu) = \int (x \otimes x)\dmu(x).$$

We occasionally use the notation $f = O(g)$ to mean that there exists an absolute constant $C > 0$ such that $f \le Cg$. In addition, we sometimes use $C$ to denote a positive absolute constant whose value may change from line to line.

\subsection*{$1$-log-concave measures}
Let $\mu$ be an even $1$-log-concave measure. As our proofs proceed via the analysis of a differential operator defined on $L^2(\mu)$, it is simplest to immediately reduce to the case that the density of $\mu$ is well-behaved.

Following \cite{K23}, we say that $\mu$ is a \textit{regular} $1$-log-concave measure if $\mu$ has a density $\rho$ which is everywhere twice-differentiable, strictly positive on $\mathbb R^n$, and such that $V = -\log\rho$ satisfies the conditions
$$I_n \le \nabla^2 V(x) \le C I_n$$
for some $C > 0$ and all $x \in \mathbb R^n$, in the sense of symmetric matrices.

We claim that it is sufficient to prove Theorems \ref{thm:cov_stability} and \ref{thm:strong_cov_stability} under the assumption that $\mu$ is regular. This follows from an entirely standard approximation argument which we sketch only briefly. The main point is to show that if $d\mu = \rho\,dx$ is an even $1$-log-concave measure then there exists a sequence of regular even $1$-log-concave measures $d\mu_k = \rho_k\,dx$ with $\rho_k \to \rho$ pointwise, and such that $|\rho_k|, |\rho| \le C e^{-c|x|}$ for some $c, C > 0$. This implies by the dominated convergence theorem that $\Cov(\mu_n) \to \Cov(\mu)$, so if each $\mu_k$ satisfies the conclusions of Theorems \ref{thm:cov_stability} and \ref{thm:strong_cov_stability} then so does $\mu$. 

To find a sequence of regular measures converging to $\mu$, one convolves $\mu$ with a Gaussian of small variance to smooth it out, and then multiplies the density of the resulting measure $\nu$ by $e^{-c |x|^2/2}$, for some small $c > 0$, to obtain a $1$-log-concave measure. We refer to \cite[\S 2, \S 5]{K23} for a fuller treatment of a similar approximation argument. 

Thus, from now on, we fix an even, regular $1$-log-concave probability measure $\mu$ on $\mathbb R^n$, and write $d\mu = e^{-V(x)}\,dx$. 

We let $L = L_\mu$ be the Laplacian associated to $\mu$, which is the second-order elliptic differential operator satisfying the integration by parts formula
$$\int (Lf) g\dmu = \int \langle \nabla f, \nabla g\rangle\dmu,$$
i.e., $Lf = -\Delta f + \nabla V \cdot \nabla f$. $L$ is an (unbounded, essentially) self-adjoint, positive-semidefinite operator on $L^2(\mu)$, whose kernel is the space of constant functions. (We remark that $L_\mu$ is often defined with the opposite sign, as a \textit{negative}-semidefinite differential operator with second-order term $\Delta$; our convention for $L_\mu$ avoids a proliferation of minus signs in many formulas below.)

As $\mu$ is $1$-log-concave, $L_\mu$ has a discrete spectrum (see \cite[Appendix]{KP23}); moreover, all eigenfunctions of $L_\mu$ decay exponentially \cite{K23}. This suffices to justify our free use of spectral decomposition and integration by parts arguments below, as the dedicated reader may verify.

Modulo these technical issues, the results we shall need from the theory of (even) $1$-log-concave measures are quite easily proven, and thus we shall recall their proofs, which will also motivate our own arguments, rather than simply citing them from the literature.

We write
$$L_0^2(\mu) = \left\{h \in L^2(\mu): \int h\dmu = 0\right\}.$$
for the space of mean-zero square-integrable functions. For $k \in \mathbb N$ we set $$\|f\|_{H^k(\mu)}^2 = \int f (L^k f)\dmu,$$
and we also set
$$\|f\|_{H^{-1}(\mu)} = \sup \left\{\int fg\dmu: \|g\|_{H^1(\mu)} \le 1\right\}.$$
Note that $\|f\|_{H^{-1}(\mu)} = \infty$ if $f \in L^2(\mu) \backslash L^2_0(\mu)$.

The main tool in our approach to the (B)-theorem is the Bochner fomula:

\begin{proposition}[Bochner formula]\label{prop:bochner_formula} For all $f \in H^2(\mu)$, 
\begin{equation}\label{eq:bochner_formula}
    \int (Lf)^2\dmu = \int  \norm{\nabla^2 f}_{HS}^2 \dmu + \int \langle \nabla^2 V \cdot \nabla f, \nabla f\rangle\dmu.
\end{equation}
\end{proposition}

The Bochner formula follows directly from the algebraic relation $\nabla L f = L \nabla f + \nabla^2 V \cdot \nabla f$ and integration by parts. 

As 
$$\int \langle \nabla^2 V \cdot \nabla f, \nabla f\rangle\dmu \ge \int \langle \nabla f, \nabla f\rangle\,d\mu = \|f\|_{H^1(\mu)},$$
by $1$-log-concavity, \eqref{eq:bochner_formula} immediately yields that $\|f\|_{H^2(\mu)} \ge \|f\|_{H^1(\mu)}$ for any $f \in H^2(\mu)$. Applying this to any eigenfunction $\varphi$ of $L$ with eigenvalue $\lambda$ we obtain $\lambda^2 \ge \lambda$, i.e., either $\lambda = 0$ (in which case $\varphi$ must be constant) or $\lambda \ge 1$.

That is, writing
$$\lambda_1 = \inf_{\varphi \in L^2_0(\mu)} \frac{\|\varphi\|_{H^1(\mu)}^2}{\|\varphi\|_{L^2(\mu)}^2}$$ 
for the spectral gap of $L_\mu$, one has $\lambda_1 \ge 1$. By substituting, for any function $f$, the function $g = f - \int f\dmu \in L^2_0(\mu)$, one obtains the Poincaré inequality for $1$-log-concave measures: for any $f$, one has 
\begin{equation}\label{eq:poincare}
\Var_\mu(f) = \int f^2\dmu - \left(\int f\dmu\right)^2 \le \int |\nabla f|^2\dmu.
\end{equation}
Note that by applying the Poincaré inequality to linear functions, one obtains in particular that $\Cov(\mu) \le I$.

\begin{proposition}[{{Even Poincaré inequality; cf. \cite{CFM03}; \cite[Theorem 3]{BC13}}}]\label{prop:even_poinc} Let $\varphi \in L^2_0(\mu)$ be an even function satisfying $L \varphi = \lambda \varphi$. Then $\lambda \ge \lambda_1 + 1$.
\begin{proof}Normalize $\|\varphi\|_{L^2(\mu)} = 1$, and apply Bochner's formula to $\varphi$:
\begin{align*}
\lambda^2 &= \int (L\varphi)^2\dmu = \int  \norm{\nabla^2 \varphi}_{HS}^2\dmu + \int \langle \nabla^2 V \cdot \nabla f, \nabla f\rangle\dmu \\
&\ge \int  \norm{\nabla^2 \varphi}_{HS}^2 \dmu + \int |\nabla \varphi|^2\dmu = \int  \norm{\nabla^2 \varphi}_{HS}^2 \dmu + \lambda
\end{align*}
To treat the first term on the RHS, note that each partial derivative $\partial_i \varphi$ is odd and in particular mean-zero, so we may apply the Poincaré inequality \eqref{eq:poincare} with $f = \partial_i \varphi$ to obtain 
$$\|\nabla \partial_i \varphi\|_{L^2(\mu)}^2 \ge \lambda_1 \|\partial_i \varphi\|_{L^2(\mu)}^2.$$
Summing over $i$ we obtain 
$$\int  \norm{\nabla^2 \varphi}_{HS}^2 \dmu \ge \lambda_1 \|\nabla \varphi\|_{L^2(\mu)}^2 = \lambda_1 \lambda,$$
and substituting in the above we obtain $\lambda^2 \ge \lambda_1 \lambda + \lambda$, i.e., $\lambda \ge \lambda_1 + 1$.
\end{proof}
\end{proposition}

As $\mu$ is even, $L_\mu$ commutes with reflection through the origin and hence the orthogonal decomposition of $L^2(\mu)$ into the subspaces of even and odd functions is stable under $L_\mu$, so every even $f \in L^2(\mu)$ decomposes as a linear combination of even eigenfunctions of $L_\mu$. Hence, letting 
$$\lambda_1^e = \inf \left\{\frac{\|f\|_{H^1(\mu)}^2}{\|f\|_{L^2(\mu)}^2}: \text{$f \in L^2_0(\mu)$ even}\right\}$$
denote the spectral gap of $L$ on the subspace of even functions, Proposition \ref{prop:even_poinc} implies that $\lambda_1^e \ge \lambda + 1 \ge 2$, and hence, if $f \in L^2(\mu)$ is even, then it satisfies the strengthened Poincar\'e inequality
\begin{equation}\label{eq:even_poincare}
\Var_\mu(f) \le \frac{1}{2} \int |\nabla f|^2\dmu.
\end{equation}

Finally, by applying spectral decomposition, it is easily seen that $\left.L\right|_{L^2_0(\mu)}$ has a well-defined, bounded inverse $L^{-1}: L^2_0(\mu) \to L^2_0(\mu)$; moreover, one has $\|L^{-1} f\|_{H^k(\mu)} = \|f\|_{H^{k - 2}(\mu)}$ for all $k$.

\section{The initial stability estimate}\label{sec:init_stability}

From now on, we fix an even, regular $1$-log-concave measure $\mu$, as in \S \ref{sec:prelims}. We begin with a stability version of the Poincaré inequality for $\mu$.

\begin{lemma}\label{lem:odd_stability}
Let $\varphi \in L^2_0(\mu)$ be a smooth function such that $\int |\nabla \varphi|^2 \dmu - \int \varphi^2 \dmu \le \delta$. Then, letting 
$$\theta = \int \varphi(x) x \dmu(x) \in \mathbb R^n,$$ 
we have:
\begin{equation}
\int \abs{\nabla \varphi - \theta}^2 \dmu \le 2\delta.
\end{equation}
\end{lemma}

\begin{proof}
Let $u$ be a smooth function such that $Lu = \varphi$. Then, by Bochner's formula \eqref{eq:bochner_formula}, we have:
\begin{align*}
\int \varphi^2 \dmu &= 2\int \varphi (Lu)\dmu - \int (Lu)^2\dmu \\
&= 2 \int \nabla\varphi \cdot \nabla u \dmu - \int (Lu)^2\dmu \\
&\le 2 \int \nabla\varphi \cdot \nabla u \dmu - \left(\int \norm{\nabla^2 u}_{HS}^2\,\dmu + \int |\nabla u|^2\dmu\right) \\ 
&= -\int \abs{\nabla \varphi - \nabla u}^2 \dmu + \int \abs{\nabla \varphi}^2 \dmu - \int \norm{\nabla^2 u}_{HS}^2 \dmu
\end{align*}
Rearranging and using the assumption $\int |\nabla \varphi|^2\dmu - \int \varphi^2\dmu \le \delta$, we obtain
\begin{equation} \label{eq:star1}
\int \norm{\nabla^2 u}_{HS}^2 \dmu + \int \abs{\nabla \varphi - \nabla u}^2 \dmu  \le \delta.
\end{equation}

Observe that
$$
\int \partial_i u \dmu = \int \langle \nabla u, \nabla x_i\rangle\dmu(x) = \int (Lu) x_i \dmu(x) = \int \varphi(x) x_i \dmu(x) = \theta_i,$$
i.e. $\theta = \int \nabla u \dmu$. Applying the Poincaré inequality \eqref{eq:poincare} coordinate-by-coordinate to the mean-zero function $\nabla u - \theta$  and using \eqref{eq:star1}, we obtain:
\begin{equation}\label{eq:poinc_nabla_u}
\int \abs{\nabla u - \theta}^2 \dmu \le \int \norm{\nabla^2 u}_{HS}^2 \dmu.
\end{equation}

Combining \eqref{eq:poinc_nabla_u} with \eqref{eq:star1} we thus obtain
\begin{align*}
\int \abs{\nabla \varphi - \theta}^2 \dmu &\le 2 \left( \int \abs{\nabla \varphi - \nabla u}^2 \dmu + \int \abs{\nabla u - \theta}^2 \dmu \right) \\ 
&\le 2\left( \int \abs{\nabla \varphi - \nabla u}^2 \dmu + \int  \norm{\nabla^2 u}_{HS}^2 \dmu \right) \le 2\delta,
\end{align*}
as claimed.
\end{proof}

\begin{proposition}\label{prop:even_stability}
Let $f \in L^2_0(\mu)$ be a smooth, even function such that 
$$\frac{1}{2}\int \abs{\nabla f}^2 \dmu - \int f^2 \dmu \le \delta.$$
Let $u$ satisfy $Lu = f$, and set $T = \int (\nabla u \otimes x) \dmu(x)$. Then
\begin{equation}\label{eq:nablaf_T} \int \abs{\nabla f - 2Tx}^2 \dmu \le 20\delta,
\end{equation}
and also
\begin{equation}\label{eq:nablaf_nablau}
\int \abs{\nabla f - 2\nabla u}^2 \dmu \le 2\delta
\end{equation}
and 
\begin{equation}\label{eq:nabla2_u}
\int \norm{\nabla^2 u}_{HS}^2 \dmu - \int \abs{\nabla u}^2 \dmu \le \delta.
\end{equation}
\end{proposition}

\begin{proof}
Using Bochner's formula as before, we have
\begin{align*}
\int f^2 \dmu &\le 2\int \inp{\nabla f}{\nabla u} \dmu - \int \norm{\nabla^2 u}_{HS}^2 \dmu - \int \abs{\nabla u}^2 \dmu \\
&= - \left( \int \norm{\nabla^2 u}_{HS}^2 \dmu - \int \abs{\nabla u}^2 \dmu \right) + 2\int \inp{\nabla f}{\nabla u} \dmu - 2\int \abs{\nabla u}^2 \dmu \\
&= - \left( \int \norm{\nabla^2 u}_{HS}^2 \dmu - \int \abs{\nabla u}^2 \dmu \right) - \frac{1}{2} \int \abs{\nabla f - 2\nabla u}^2 \dmu + \frac{1}{2} \int \abs{\nabla f}^2 \dmu,
\end{align*}

Applying the assumption $\frac{1}{2}\int \abs{\nabla f}^2 \dmu - \int f^2 \dmu \le \delta$ and rearranging, we obtain:
\begin{equation}\label{eq:star2}
\max \left\{ \int \norm{\nabla^2 u}_{HS}^2 \dmu - \int \abs{\nabla u}^2 \dmu,\, \frac{1}{2} \int \abs{\nabla f - 2\nabla u}^2 \dmu \right\} \le \delta,
\end{equation}
which in particular yields \eqref{eq:nablaf_nablau} and \eqref{eq:nabla2_u}. 

Define $\delta_i \ge 0$ as
$$\delta_i = \int |\nabla \partial_i u|^2 \dmu - \int (\partial_i u)^2 \dmu,$$
so that $\sum_{i = 1}^n \delta_i \le \delta$, and let $t_i = \int x\, \partial_i u\dmu(x) \in \mathbb R^n$, the $i$th row of the matrix $T = \int (\nabla u \otimes x)\dmu(x)$. Applying Lemma \ref{lem:odd_stability} to $\partial_i u$ for each $i$, we obtain
$$
\int |\nabla \partial_i u - t_i|^2 \dmu \le 2\delta_i,
$$
and summing up, we obtain
$$\int \norm{\nabla^2 u - T}_{HS}^2 \dmu \le 2\sum_{i=1}^n \delta_i \le 2\delta.$$
By the Poincaré inequality \eqref{eq:poincare}, this yields
$$\int \abs{\nabla u - Tx}^2 \dmu(x) \le 2\delta.$$

Finally, employing \eqref{eq:star2} again:
$$
\int \abs{\nabla f - 2Tx}^2 \dmu \le 2 \left( \int \abs{\nabla f - 2\nabla u}^2 \dmu + 4\int \abs{\nabla u - Tx}^2 \dmu \right) \le 20\delta.$$
\end{proof}

We can now prove Theorem \ref{thm:strong_cov_stability}, which we restate here for ease of reference:

\begin{manualtheorem}{\ref{thm:strong_cov_stability}} Let $\mu$ be an even $1$-log concave measure, $A \in \mathbb R^{n \times n}$ a symmetric matrix, and let $f_A(x) = \langle x, Ax\rangle$. If 
\begin{equation}
\frac{1}{2}\mathbb E_\mu[|\nabla f_A|^2] - \Var_\mu(f_A)  \le \delta,
\end{equation}
then
$$\|A(I - \Sigma) \Sigma^{1/2}\|_{HS}^2 \le 11\delta.$$
\end{manualtheorem}

\begin{proof} Apply Proposition \ref{prop:even_stability} to $f_A$, and let $T$ be as defined there. By the proposition, $\int \abs{\nabla f - 2Tx}^2 \dmu \le 20\delta$. But 
\begin{align*}
\int \abs{\nabla f - 2Tx}^2\,d\mu &= 4 \int \abs{(A-T)x}^2 \dmu \\
&= 4 \sum_{i,j,k} (A-T)_{ij}(A-T)_{ik} \int x_j x_k \dmu = 4 \inp{(A-T)\Sigma}{A-T},
\end{align*}
so we obtain $\norm{(A-T)\Sigma^{1/2}}_{HS}^2 \le 5\delta$. 

Next, we have
$$\frac{1}{2} \int (\nabla f \otimes x)\dmu = \int (Ax \otimes x)\dmu =  A\Sigma,$$
so 
$$A\Sigma - T = \frac{1}{2} \int ((\nabla f - 2\nabla u) \otimes x)\dmu.$$
Thus, for any choice of orthonormal basis $\theta_1, \ldots, \theta_n$ of $\mathbb R^n$, we may write $$\|A\Sigma - T\|_{HS}^2 = \sum_{i, j = 1}^n \langle \theta_i, (A\Sigma - T)\theta_j\rangle^2 = \frac{1}{4}\sum_{i, j = 1}^n \left(\int \langle \theta_i, \nabla f - 2\nabla u\rangle \langle x, \theta_j\rangle\,d\mu\right)^2.$$
Choose $\theta_1, \ldots, \theta_n$ to be an orthonormal basis of eigenvectors of $\Sigma$, which in particular satisfy 
$$\int \langle x, \theta_i\rangle \langle x, \theta_j\rangle\,d\mu = 0$$
for $i \neq j$ and
$$\int \langle x, \theta_i\rangle^2\,d\mu = \langle \theta, \Sigma \theta\rangle \le 1.$$
For $i \in \{1, \ldots, n\}$, write $g_i(x) = \langle \theta_i, \nabla f - 2\nabla u \rangle$, $\ell_i(x) = \langle x, \theta_i\rangle$. Then as the $\ell_i$ are orthogonal in $L^2(\mu)$ and satisfy $\|\ell_i\|_{L^2(\mu)} \le 1$, we have for any $i$ that
$$\sum_{j = 1}^n \left(\int \langle \theta_i, \nabla f - 2\nabla u\rangle \langle x, \theta_j\rangle\,d\mu\right)^2 = \sum_{j = 1}^n \langle g_i, \ell_j\rangle_{L^2(\mu)}^2 \le \|g_i\|_{L^2(\mu)}^2,$$
and hence, applying \eqref{eq:nablaf_nablau}, that
\begin{align}
    \|A\Sigma - T\|_{HS}^2 &= \frac{1}{4}\sum_{i, j = 1}^n \left(\int \langle \theta_i, \nabla f - 2\nabla u\rangle \langle x, \theta_j\rangle\,d\mu\right)^2 \nonumber \\
    &\le \frac{1}{4}\sum_{i = 1}^n \int \langle \theta_i, \nabla f - 2\nabla u\rangle^2\,d\mu = \frac{1}{4}\int |\nabla f - 2\nabla u|^2\,d\mu \le \frac{\delta}{2}. \label{eq:sigma_t_diff}
\end{align}

As $\Sigma \le I$, we also have $\norm{(A\Sigma - T) \Sigma^{1/2}}_{HS} \le \norm{A\Sigma - T}_{HS} \le \sqrt{\delta/2}$. Putting everything together, we obtain
$$
\norm{A(I-\Sigma)\Sigma^{1/2}}_{HS}^2 \le 2 \left( \norm{(A-T)\Sigma^{1/2}}_{HS}^2 + \norm{(T-A\Sigma)\Sigma^{1/2}}_{HS}^2 \right) \le 11\delta,
$$
as claimed.
\end{proof}

\section{Improving the bounds in the large-variance regime}\label{sec:sharp_impr}
To determine whether Theorem \ref{thm:strong_cov_stability} is sharp in the case $A = I$, we can test it in the case in which $\mu = \mathcal N(0, v)$ for some $v \le 1$. In this case, it is well-known that $\mathbb E_\mu[x^4] = 3v^2$ and hence, letting $f(x) = x^2$, one has 
$$\frac{1}{2}\mathbb E[(f')^2] - \Var_\mu(f) = \frac{1}{2} \mathbb E[4x^2] - \mathbb E[x^4] + \mathbb E[x^2]^2 = 2v - 3v^2 + v^2 = 2v (1 - v).$$
Hence, the condition $\frac{1}{2}\mathbb E[(f')^2] - \Var_\mu(f) \le 2\delta$ implies $v (1 - v) \le \delta$, so either $v = O(\delta)$ or $1 - v = O(\delta)$. This is precisely what Theorem \ref{thm:cov_stability}, which we have yet to prove, predicts. Theorem \ref{thm:strong_cov_stability}, however, yields only that 
$$\|\Sigma^{1/2} (I - \Sigma)\|_{HS}^2 = O(\delta),$$ which reduces in the one-dimensional case to $v (1 - v)^2 = O(\delta)$, i.e., either $v = O(\delta)$ or $1 - v = O(\sqrt\delta)$.

The proof of Theorem \ref{thm:cov_stability} uses the following result, which is a generalization of a result of Klartag \cite[Proposition 2.3]{K23}. (Klartag's result is the special case of Proposition \ref{prop:pref_direction} in which $f$ is an eigenfunction.)

\begin{proposition}\label{prop:pref_direction} Suppose there exists $f \in L^2_0(\mu)$ with $\|f\|_2 = 1$, $\|\nabla f\|_2^2 \le 1 + \delta$, and let $\theta$ be a unit vector in the direction of $\int xf\dmu$. Then 
$$\int \langle x, \theta\rangle^2\dmu \ge 1 - 2\delta.$$
\end{proposition}
\begin{proof}Applying the Bochner formula to $u = L^{-1} f$ and the Poincaré inequality, and recalling that $\int \nabla u\dmu = \int xf\dmu$, we obtain 
\begin{align*}
1 = \|f\|_{L^2(\mu)}^2 = \int (Lu)^2\dmu &\ge \int  \norm{\nabla^2 u}_{HS}^2 \dmu + \int |\nabla u|^2\dmu \\
&\ge \left(\int |\nabla u|^2 \dmu - \left|\int \nabla u\dmu\right|^2 \right) + \int |\nabla u|^2\dmu \\ 
&= 2\|f\|^2_{H^{-1}(\mu)} - \left|\int x f\dmu\right|^2.
\end{align*}
By Cauchy-Schwarz, 
$$\|f\|_{H^{-1}(\mu)}^2 \ge \frac{\|f\|_{L^2(\mu)}^4}{\|f\|_{H^1(\mu)}^2} \ge \frac{1}{1 + \delta},$$
so 
$$1 \ge \frac{2}{1 + \delta} - \left|\int xf\dmu\right|^2,$$
i.e., 
$$\left|\int xf\dmu\right|^2 \ge \frac{1 - \delta}{1 + \delta} \ge 1 - 2\delta.$$

Now note that $\left|\int xf\dmu\right| = \int \langle x, \theta\rangle f\dmu$, and apply Cauchy-Schwarz to conclude that
$$\int \langle x,\theta\rangle^2\dmu(x) = \int \langle x,\theta\rangle^2\dmu(x) \int f^2\dmu \ge \left|\int \langle x,\theta\rangle f(x)\dmu(x)\right|^2 \ge 1 - 2\delta.$$
\end{proof}

We may now conclude the proof of Theorem \ref{thm:cov_stability}, which we restate here:

\begin{manualtheorem}{\ref{thm:cov_stability}} Let $\mu$ be an even $1$-log-concave measure with covariance matrix $\Sigma$, let $f(x) = |x|^2$ and suppose that 
\begin{equation}\label{eq:spectral_stability_rst}
\frac{1}{2} \mathbb E[|\nabla f|^2] - \Var_\mu(f) \le \delta
\end{equation}
for $\delta \in (0, c)$, where $c > 0$ is an absolute constant. Then
\begin{equation}\label{eq:cov_stability}
\Tr(\Sigma (I - \Sigma)) \le C\delta
\end{equation}
for some universal constant $C > 0$.
\end{manualtheorem}

\begin{proof}
Suppose $\mu$ satisfies \eqref{eq:spectral_stability_rst}, and let $\sigma_1 \ge \cdots \ge \sigma_n$ denote the eigenvalues of $\Sigma = \Cov(\mu)$. By Theorem \ref{thm:strong_cov_stability}, we have 
$$\sum_{i = 1}^n \sigma_i (1 - \sigma_i)^2 \le C \delta,$$
which implies that 
$$\sum_{\sigma_i \le \frac{1}{2}}  \sigma_i (1 - \sigma_i)\le 2C \delta.$$
To prove Theorem \ref{thm:cov_stability}, it suffices to show that 
$$\sum_{\sigma_i \ge \frac{1}{2}}  (1 - \sigma_i) \le C \delta.$$

Let $k = \max \{j: \sigma_j \ge \frac{1}{2}\}$, and let $V$ be the span of the eigenvectors of $\Sigma$ corresponding to $\sigma_1, \ldots, \sigma_k$.

For a unit vector $\theta \in V$, consider $\ell_\theta(x) = \langle x, \theta\rangle$, so that $\|\ell_\theta\|_{L^2(\mu)}^2 = \langle \theta, \Sigma \theta\rangle \ge \frac{1}{2}$. Letting $u = L^{-1} f$ as before, recall that by \eqref{eq:nablaf_nablau}, we have $\|\nabla f - 2\nabla u\|_{L^2(\mu)}^2 \le 2\delta$, i.e., $\|2x - 2\nabla u\|_{L^2(\mu)}^2 \le 2\delta$, and in particular $\|\ell_\theta - \partial_\theta u\|_{L^2(\mu)}^2 \le \frac{\delta}{2}$. This implies that 
$$\|\partial_\theta u\|_{L^2(\mu)}^2 \ge \left(\|\ell_\theta\| - \sqrt{\delta/2}\right)^2 \ge \frac{1}{3}.$$

Let $\theta_1, \ldots, \theta_k$ be an orthonormal basis of $V$ (to be chosen later). We have that 
$$\int  \norm{\nabla^2 u}_{HS}^2 \dmu - \int |\nabla u|^2\dmu \le \delta$$ 
by \eqref{eq:nabla2_u} and hence, in particular, 
$$\sum_{i = 1}^k (\|\nabla \partial_{\theta_i} u\|^2 - \|\partial_{\theta_i} u\|^2) \le \delta.$$
Write $\delta_i = \|\nabla \partial_{\theta_i} u\|^2 - \|\partial_{\theta_i} u\|^2$, and define $\tilde u_i = \frac{\partial_{\theta_i} u}{\|\partial_{\theta_i} u\|}$. 

By the above, we have $\|\tilde u_i\|_2^2 = 1$, $\|\nabla\tilde u_i\|_2^2 \le 1 + 3\delta_i$, and by Proposition \ref{prop:pref_direction}, this implies that if $\eta_i$ is a unit vector in the direction of $\int x \,\tilde u_i\dmu$ -- which is the same direction as that of $\int x\,\partial_{\theta_i} u\dmu$ -- then we have 
\begin{equation}\label{eq:large_var}
\langle \eta_i, \Sigma \eta_i\rangle = \Var_\mu(\langle x, \eta_i\rangle) \ge 1 - 6\delta_i.    
\end{equation}

Now let $T = \int (x \otimes \nabla u)\dmu$ as before. By \eqref{eq:sigma_t_diff}, $\|T - \Sigma\|_{HS} = O(\sqrt\delta)$ and hence, as $\left.\Sigma\right|_{V} \ge \frac{1}{2}$, $T$ is injective on $V$ and so $\dim TV = k$. The singular value decomposition implies the existence of an orthonormal basis $\theta_1, \ldots, \theta_k$ of $V$ such that the vectors $T\theta_1, \ldots, T\theta_k$ are orthogonal. But $T\theta_i = \int x \,\partial_{\theta_i} u \dmu$, so letting $\eta_i = \frac{T\theta_i}{|T\theta_i|}$ and applying \eqref{eq:large_var}, we have 
\begin{align*}
    \Tr(\left.\Sigma\right|_{TV}) &= \sum_{i = 1}^k \langle \eta_i, \Sigma \eta_i\rangle \ge  \sum_{i = 1}^k (1 - 6\delta_i) = k - 6\delta.
\end{align*}

The Wielandt minimax principle states that if $\lambda_1, \ldots, \lambda_j$ are the largest $j$ eigenvalues of a symmetric matrix $M \in \mathbb R^{n \times n}$, then 
$$\lambda_1 + \cdots + \lambda_j = \sup \{\Tr(\left.M\right|_W): \text{$W \subset \mathbb R^n$ a linear subspace of dimension $j$}\}.$$
Applying this to $\Sigma$, we obtain that the $k$ leading eigenvalues of $\Sigma$, namely $\sigma_1, \ldots, \sigma_k$, satisfy 
$$\sigma_1 + \cdots + \sigma_k \ge k - 6\delta,$$
i.e.,
$$\sum_{i = 1}^k (1 - \sigma_i) \le 6\delta,$$
concluding the proof.
\end{proof}

\bibliographystyle{plain}     
\bibliography{references}

\end{document}